\documentclass[a4paper,10pt]{article}

\usepackage[]{graphicx}
\usepackage{amsmath,amsfonts,amssymb,amsthm}
\usepackage{bm}

\theoremstyle{plain}
\newtheorem{proposition}{Proposition}
\newtheorem{theorem}{Theorem}

\theoremstyle{definition}
\newtheorem{definition}{Definition}
\theoremstyle{remark}

\newcommand{\CDOT}{{\boldsymbol\cdot}}
\newcommand{\STAR}{{\odot}}

\newcommand{\dfin}[2]{#1\in\mathcal{A}^{#2}}
\newcommand{\hodgeop}{\star}
\newcommand{\Laplace}{\triangle}

\title{Riemannian Geometry Based on the Takagi's Factorization of the Metric Tensor}

\author{Juan M\'endez\\
 Edificio de Inform\'atica y Matem\'aticas\\
            Universidad de  Las Palmas de Gran Canaria\\
            35017 Las Palmas, Spain\\
            email: {juan.mendez@ulpgc.es}
}

\begin{document}

\maketitle

\begin{abstract}
The Riemannian geometry is one of the main theoretical pieces in Modern Mathematics and Physics. The study of Riemann Geometry in the relevant literature is performed by using a well defined analytical path. Usually it starts from the concept of metric as the primary concept and by using the connections as an intermediate geometric object, it is achieved the curvature and its properties.
This paper presents a different analytical path to analyze the Riemannian geometry. It is based on a set of intermediate geometric objects obtained from the Takagi's factorization of the metric tensor. These intermediate objects allow a new viewpoint for the analysis of the geometry, provide conditions for the curved vs. flat manifolds, and also provide a new decomposition of the curvature tensor in canonical parts, which can be useful for Theoretical Physics.
\end{abstract}

\section{Introduction}

The Riemannian geometry\cite{doCarmo:1992,Jost:2011,Kobayashi:1996}  has been the main theoretical contribution that allowed the development of non-Euclidean geometries in the late nineteenth century. Also in the twentieth century it has been the main tool that has allowed the development of the General Relativity in which the Geometry and Gravitation have been unified into an elegant theoretical framework, and until today without  experimental discrepancies. Today, the Riemannian geometry remains as a non-exhausted source for advanced studies for disciplines as Geometry and Theoretical Physics.

The exact solution of the Einstein equation\cite{Stephani:2009}, $\mathbf{G}=8\pi \mathbf{T}$, is
one of the main fully unsolved problems in modern Theoretical Physic. It has been solved in some special cases but the research community is far to have a methodology to provide solutions for general cases, even though the great activity involved. Nowadays, it remains being a motivational field.

The Einstein equation involves the Einstein tensor that is an geometric object obtained from the Riemann curvature after some contractions, and the stress-energy tensor that is a physic object. The Riemann curvature has focused many of the research studies in its properties, decomposition and factorization in canonical types. This is implicitly the goal of this paper, but it is not addressed directly, rather it is addressed from a lower level. The contribution of this paper is to study the Riemannian geometry from a different viewpoint as how is presented and analyzed in the reference literature. The analytical path usually presents the metric as the primary concept from which is obtained in successive steps the connection, the curvature, Ricci and Einstein tensors. All these, which can be describe as the concepts, entities or abstract objects of the geometry, also can be obtained based on a different analytical path by using a specific tool as is the Takagi's decomposition or factorization of the metric tensor.

The Takagi's decomposition of the metric tensor generates a set of intermediate objects that
allows a different path in the geometry analysis. One of the advantages is that allows a clarifying  use of some Topological concepts to classify the manifolds as curved or flat by using a different test that the curvature tensor.. This proposal is less economical in the number of intermediate objects, but the main advantage is that provides a different, non-better, viewpoint of the Riemannian geometry.

Matrix factorization or decomposition\cite{Horn:1985}, as LU or Cholesky,  has been used in many areas of the Mathematics to solve problems involving matrix calculus. Matrix factorization allows to express a matrix in  some normalized expression that simplifies the procedures involved in matrix theory, algorithms and computational tasks.  Usually, the matrix factorization provides some advantages for reducing the complexity. Perhaps the most active use of matrix factorization is in the High Performance Computing arena because its extensive use in the solution of linear equation systems using high parallel computers. However, far to that economical utility, the matrix factorization can also provide an utility in analysis of abstract problems, how is the case of use in this paper.

The plan of this paper is the following, Section \ref{sec:takagi} presents the Takagi's factorization of symmetric matrices and the definition of a new operator required to compactly express some vector equations. Section \ref{sec:takagimetric} presents the Takagi's factorization of the metric tensor. Section \ref{sec:objects} presents the intermediate geometric objects obtained from the factorization of the metric tensor and how the traditional objects of the Riemann geometry, as the Levi-Civita connection, the curvature, Ricci and Einstein tensors are obtained from these intermediate objects. The paper ends with the Conclusion Section and References.

\section{Takagi's Factorization of Symmetric Matrices}\label{sec:takagi}

The Takagi's factorization of a symmetric matrices is one of the matrix factorization procedures related to the eigenvalue decomposition. Although in this paper we are only interested in real matrices, the Takagi's factorization is more general regarding complex matrices. If $\mathcal{M}_{n\times n}(\mathbb{C})$ is the space of complex $n\times n$ matrices. Let $\mathbf{W}\in \mathcal{M}_{n\times n}(\mathbb{C})$ be a symmetrical matrix, then the Takagi's factorization\cite{Horn:1985} proves that exists an unitary matrix $\mathbf{U}\in \mathcal{M}_{n\times n}(\mathbb{C})$  and a nonnegative diagonal matrix $\mathbf{\Sigma}= \mathrm{diag}(\sigma_{1},\ldots,\sigma_{n})$ such that:
\begin{equation}\label{eq:Takagi}
\mathbf{W} = \mathbf{U}\mathbf{\Sigma}\mathbf{U}^{T}
\end{equation}

where the elements of $\mathbf{\Sigma}$ are the nonnegative square roots of the eigenvalues of $\mathbf{W}\mathbf{W}^{H}$ and the elements of $\mathbf{U}$ are an orthogonal set of the corresponding eigenvectors.
An equivalent expression for the Takagi's factorization\cite{Horn:1985}[Co\-ro\-llary 4.4.5] can be expressed based on non-unitary matrix $\mathbf{V}$ as follows:
\begin{equation}\label{eq:Takagi2}
\mathbf{W} = \mathbf{V}\mathbf{V}^{T}
\end{equation}

Although $\mathbf{W}$ be a real matrix, $\mathbf{W}\in \mathcal{M}_{n\times n}(\mathbb{R})$, the matrix $\mathbf{V}$ can be complex, $\mathbf{V}\in \mathcal{M}_{n\times n}(\mathbb{C})$, that depends greatly on the sign of the diagonal elements of the matrix $\mathbf{W}$.
We can analyze the matrix $\mathbf{V}$ by studding it as a decomposition in a set of vector rows or as a set of vector columns. In the first approach, the analysis is based on row-vectors, each row, eg. the $a$-nth , corresponds to a vector $\mathbf{R}^{(a)}$ containing the elements: $R^{(a)}_{b}= V_{ab}$. In this case the elements of matrix $\mathbf{W}$ can be expressed as:
\begin{equation}\label{eq:dotproduct}
W_{ab} = \sum_{c=1}^{n} V_{ac}(V^{T})_{cb} = \sum_{c=1}^{n} V_{ac}V_{bc} = \sum_{c=1}^{n} R^{(a)}_{c}R^{(b)}_{c} = \mathbf{R}^{(a)} \CDOT \mathbf{R}^{(b)}
\end{equation}

where we have used the dot product of two vectors, expressed by the use of  the $\CDOT$ operator, very common in elementary vector and matrix Algebra. Although it is obvious and seems too much elementary, we must remember that it is only an abstract way to express an hidden sum-of-products in the component domain.

The second approach is to study the matrix $\mathbf{V}$ as a set of column vectors, such that the column $a$-nth corresponds to the vector: $\mathbf{C}^{(a)}$, being: $C^{(a)}_{b} = V_{ba}$. The elements of matrix $\mathbf{W}$ can be expressed as:
\begin{equation}\label{eq:setproduct}
W_{ab} = \sum_{c=1}^{n} V_{ac}V_{bc} = \sum_{c=1}^{n} C^{(c)}_{a}C^{(c)}_{b} = \mathbf{C}_{a} \STAR \mathbf{C}_{b}
\end{equation}

where we have introduced a new operator $\STAR$, which is an hidden sum-of-products in the  vector set, while the operator $\CDOT$ is an hidden sum-of-products in the vector components. The operators $\STAR$ and $\CDOT$ are suitable abstractions to simplify the mathematical expressions related to vector and matrix operations. The algebra of this operator is simple.

\begin{definition}[Set Product]
Let $A_{I}\in\mathbb{C}^{n}$ and $B_{I}\in\mathbb{C}^{n}$ two sets of vectors such as the index: $I=1,\ldots,m$  is an enumeration in the set. 
The set product, $\STAR$, is a tensor $T_{ab}$ defined as:
\begin{equation}\label{eq:setproduct2}
T_{ab}= A_{a}\STAR B_{b} = \sum_{I=1}^{m}A_{Ia}B_{Ib}
\end{equation}
\end{definition}

If $A_{Ia}$ and $B_{Ib}$ are tensors of rank $1$, then $A_{Ia}B_{Ib}$ is a tensor of rank $2$; the sum of tensors of same rank is also a tensor, and therefore $T_{ab}$ is a tensor of rank $2$. The set product is symmetric: $A_{a}\STAR B_{b} = B_{b}\STAR A_{a}$, but the tensor $T_{ab}$ is not, $A_{a}\STAR B_{b}\neq A_{b}\STAR B_{a}$. Also, the distributive property is verified: $A\STAR(B+C)=A\STAR B + A\STAR C$. It is an abstraction of a sum-of-products, thus the Leibnitz derivative rule must be used according to its definition, that is: $d(A\STAR B)= dA\STAR B + A\STAR dB$. The definition of the $\STAR$ product can be extended to tensors of higher rank as:
\begin{equation}
T_{a\ldots b\ldots}= A_{a\ldots}\STAR B_{b\ldots} = \sum_{I=1}^{m}A_{Ia\ldots}B_{Ib\ldots}
\end{equation}

If the matrix $\mathbf{W}$ is not singular, then it is also no singular the matrix $\mathbf{V}$ and therefore the vector set  $\mathbf{R}^{(a)}$ and $\mathbf{C}^{(a)}$, used in Equations (\ref{eq:dotproduct}) and (\ref{eq:setproduct}), are linearly independent; if $\sum_{a}\lambda_{a}\mathbf{C}^{(a)}=0$, it implies that: $\lambda_a=0$

\section{Factorization of the Metric Tensor}\label{sec:takagimetric}

Let $(M,g)$ be a Riemann manifold, that is, a $n$-dimensional compact, differentiable, oriented and connected manifold $M$ with a metric $g$ locally reducible to a diagonal case:
\begin{equation}
\eta=\textrm{diag}(\;\underbrace{1,\ldots,1}_{r},\underbrace{-1,\ldots,-1}_{s}\;)
\end{equation}

where $r$ and $s=n-r$ are the number of positive and negative ones respectively. If both $r$ and $s$ are non null, it  is a pseudo-Riemann, or semi-Riemann, manifold with indefinite metric, while \emph{pure} Riemann manifold is a particular case that has positive defined metric with $s=0$.

Both tensors and differential forms allow the study of invariant properties in the manifold and are widely used on this paper. Let $\mathcal{A}^{p}$ be the set of $p$-forms on $M$, the Hodge duality gets a linear isomorphism between $\mathcal{A}^{p}$ and $\mathcal{A}^{n-p}$. The Hodge star operator, $\hodgeop$, defines a linear map $\hodgeop : \mathcal{A}^{p} \rightarrow \mathcal{A}^{n-p}$, verifying for $\dfin{\phi}{p}$\cite{Gockeler:1989}:
\begin{equation}
\hodgeop\hodgeop  \phi = (-1)^{D(p)}\phi \qquad D(p)=p(n-p)+s
\end{equation}

The exterior derivative, that defines a linear map $d:\mathcal{A}^{p}\rightarrow\mathcal{A}^{p+1}$,  allows the definition of the coderivative  $\delta:\mathcal{A}^{k}\rightarrow\mathcal{A}^{k-1}$ defined as\cite{Gockeler:1989}:
\begin{equation}
\delta\phi = (-1)^{C(p)}\hodgeop  d\hodgeop  \phi \qquad C(p)=np+n+1+s
\end{equation}

which, similar to $dd=0$, verifies: $\delta\delta = 0$. Let $\Laplace$ be a  second order differential operator, called Laplace-Beltrami,  that map $\Laplace : \mathcal{A}^{p}\rightarrow \mathcal{A}^{p}$. It is defined as: $ \Laplace = \delta d + d \delta$. A $p$-form $\dfin{\phi}{p}$ is called harmonic if it verifies: $\Laplace \phi=0$. In \emph{pure} Riemann's manifolds\cite{Jost:2011} this implies that it is closed: $d\phi=0$ and dual-closed: $\delta\phi=0$.

In Riemann manifolds with positive signature the harmonic forms defined by: $\Laplace \phi=0$ has the solutions of a second order elliptic differential equation, while in manifolds with negative signature the solutions are of a second order hyperbolic differential equation, whose solutions are in general some type of waves.

The Riemann manifold $M$ has a metric, $g_{ab}$, and a torsion free connection, or Levi-Civita connection\cite{Jost:2011}, such as it can be expressed by using the Chistoffel symbols. In every point $p$ of the manifold, the metric tensor defines the line element: $ds^{2}=g_{ab}dx^{}dx^{b}$ or in a general case: $ds^{2}=g_{ab}\mathbf{w}^{a}\mathbf{w}^{b}$, where $\mathbf{w}^{a}$ is the dual of the basis $\mathbf{e}_{a}$ associated to the tangent space in $p$. In this point of the manifold, we can apply the Takagi's Factorization of the metric tensor, $g_{ab}$; it provides a decomposition which generates a vector set, thus the factorization extended to all the points of the manifold defines a set of vector fields. The metric tensor $\mathbf{g}$ in a point $p$ can factorized as:
\begin{equation}
\mathbf{g}(p)= \mathbf{E}(p)\mathbf{E}^{T}(p)
\end{equation}

where the matrix $\mathbf{E}(p)\in \mathcal{M}_{n\times n}(\mathbb{C})$ is: 
\begin{equation}
\mathbf{E}(p) =
\left(
  \begin{array}{cccc}
    e_{11} & e_{12} & \cdots & e_{1n} \\
    e_{21} & e_{22} & \cdots & e_{2n} \\
    \vdots & \vdots & \ddots & \vdots \\
    e_{n1} & e_{n2} & \cdots & e_{nn} \\
  \end{array}
\right)
\end{equation}

The two approaches defined in the Section \ref{sec:takagi} can be used to analyze the matrix $\mathbf{E}$. These approaches are based in the use of the operators $\CDOT$ or $\STAR$, by constructing row or column vectors. The row vectors: $\mathbf{e}_{a}$ and column vectors: $\mathbf{A}_{a}$ can be used, where $a=1,\ldots,n$. The row vectors are constructed as: $\mathbf{e}_{a}=[e_{a1},e_{a2},\ldots,e_{an}]$ and the column vectors are constructed as: $\mathbf{A}_{a}=[e_{1a},e_{2a},\ldots,e_{na}]$. In both cases each element $g_{ab}$ can be expressed as:
\begin{equation}
g_{ab} = \mathbf{e}_{a}\CDOT\, \mathbf{e}_{b} = \mathbf{A}_{a}\STAR \mathbf{A}_{b}
\end{equation}

The option based on the use of the $\CDOT$ operator defines the vector basis $\mathbf{e}_{a}$ widely used in geometric analysis. But the analytical option based on the use of the $\STAR$ operator, which is the studied in this paper, involves the use of a set of linear independent vector fields, or  1-forms, $\mathbf{A}$. This last option implies that the metric tensor is factorized as:
\begin{equation}\label{eq:main}
g_{ab} = \sum_{I=1}^{n} A_{Ia}A_{Ib}
\end{equation}

\begin{definition}[Takagi Factorization]\label{def:takagi_metric}
The metric tensor can be factorized in each point $p$ of the manifold by using a set $\{\mathbf{A}_{1}, \ldots, \mathbf{A}_{n}\}$ of  linearly independent $1$-form with components $A_{Ia}$, where the upper-case index $(I, J, \ldots)$ are the index of set enumeration, and the lower-case index $(a, b, \ldots)$ are the corresponding to the components. The $1$-forms are:
\begin{equation}\label{eq:t3}
\mathbf{A}_{I} = A_{Ia}\mathbf{w}^{a}
\end{equation}
The metric tensor can be expressed as: $g_{ab}=A_{a}\STAR A_{b}$ and the line element can be expressed as:
\begin{equation}
ds^{2}  = \sum_{I=1}^{n} \mathbf{A}_{I}\mathbf{A}_{I}
\end{equation}
\end{definition}

In the paper, the vector fields will be indistinctly used in $1$-form representation: $\mathbf{A}_{I}=A_{Ia}\mathbf{w}^{a}$ and as a vector representation: $\mathbf{A}_{I}=A_{I}^{a}\mathbf{e}_{a}$, and its tensorial derivatives, because both represent the same geometric object\cite{Stephani:2009}.

\begin{proposition}\label{prop:plainspace}
If the $1$-forms $\mathbf{A}_{I}$ are exacts: $\mathbf{A}_{I} = d\phi_{I}$, where $\phi_{I}$ is a  function or 0-form, then the metric can be reduced to a Cartesian-like one as:
\begin{equation}
ds^{2} = d\phi \STAR d\phi= \sum_{I=1}^{n} (d\phi_{I})^{2}
\end{equation}
\end{proposition}

Due that $\mathbf{A}_{I}$ can be complex, it implies that some of the terms $(d\phi_{I})^{2}$ can become negative according to the metric signature. This can be expressed as:
\begin{equation}
ds^{2} =  \sum_{I=1}^{n} \pm |d\phi_{I}|^{2}
\end{equation}

By using a suitable coordinate change such as: $d\phi_{I} \rightarrow dy^{a}$ we can obtain a Cartesian-like coordinate system with the corresponding signature, eg. one as: $(\eta_{ab})=\mathrm{diag}(1,1,1,-1)$. In this case, it is a global Cartesian-like metric, not only a locally one, and thus the manifold is flat. Therefore the non-exact property of $1$-forms $\mathbf{A}_{I}$ is the related to the curvature of the manifold.

\begin{theorem}\label{th:plainspace}
A sufficient condition for the manifold be flat is that the set of $1$-forms $\mathbf{A}_{I}$ be exact.
\end{theorem}

\begin{proposition}\label{prop:2}
Each 1-forms can be expressed according the Hodge Decom\-position\cite{Morita:2001} as: $\mathbf{A}_{I} = \mathbf{X} + \mathbf{Y}$, where $\mathbf{X}$ is a closed 1-form, $d\mathbf{X}=0$, and $\mathbf{Y}$ is a dual-closed $1$-form,  $\delta \mathbf{Y}=0$. It means that we can expressed the $1$-forms as:  $\mathbf{A}_{I} = d\phi_{I} + \mathbf{A}'_{I}$ with $\delta \mathbf{A}'_{I}=0$.
\end{proposition}

The dual-closed part in the Hodge Decomposition is the concerning to the curved manifold because the closed part alone generates a flat Cartesian-like. The Hodge Decomposition is more precise because the general decomposition is:
\begin{equation}\label{hodge}
\mathbf{A}_{I} = d\alpha + \delta\beta + \gamma
\end{equation}

where $\alpha$ is a $0$-form, $\beta$ is a $2$-form and $\gamma$ is an harmonic form characteristic of the homology class. The decomposition in the Proposition \ref{prop:2} includes the general case because $\phi_{I}=\alpha$ contains the closed part and $\mathbf{A}'_{I}=\delta\beta + \gamma$  includes the dual closed and the homology parts, which is also closed and dual closed: $d\gamma=0$ and $\delta\gamma=0$.

\begin{proposition}\label{prop:3}
It is verified that: $A_{Ic}A_{J}{}^{c} =\delta_{IJ}$.
\end{proposition}

\begin{proof}
It is verified that: $A_{a}\STAR A^{b}=\delta^{b}_{a}$. If we multiply both sides by $A_{Jb}$, then:
\begin{equation}
A_{Jb} \left(\sum_{I=1}^{n}A_{Ia}A^{b}_{I} \right ) = \delta^{b}_{a} A_{Jb} = A_{Ja}
\end{equation}

that is equivalent to:
\begin{equation}
\sum_{I=1}^{n}A_{Ia}\left(A^{b}_{I} A_{Jb}\right) = A_{Ja}
\end{equation}

and:
\begin{equation}
\sum_{I=1}^{n}A_{Ia}\left[\left(A^{b}_{I} A_{Jb}\right) -\delta_{IJ}\right]=0
\end{equation}

due to the linear independence of $A_{I}$, if $\sum_{I=1}^{n}A_{Ia}\lambda_{I}=0$, then it must be $\lambda_{I}=0$
\end{proof}

\begin{proposition}[Gauge]\label{prop:gauge}
A normalization can be use to change the $1$-form set to be dual-closed, that is $\delta \mathbf{A}_{I}=0$.
\end{proposition}
\begin{proof}
Let $\lambda$ be a 0-form used to normalize the set of $1$-forms. A change in the set of $1$-forms as: $\mathbf{A}_{I}\rightarrow \mathbf{A}_{I}+d\lambda$ implies that $d\phi_{I}+\mathbf{A}'_{I} \rightarrow d(\phi_{I} +\lambda) + \mathbf{A}'_{I}$, but this will not change the properties of the curved space. Also, in this change it is verified that: $\delta \mathbf{A}_{I} \rightarrow \delta\mathbf{A}_{I}+ \Laplace\lambda$. We can choose $\lambda$ verifying $\Laplace\lambda = - \delta\mathbf{A}_{I}$ previous to the normalization, such that the $\mathbf{A}$ set become dual closed after the normalization. The solution for the Laplacian equation $\Laplace\alpha=\beta$ is always supposed, expressed by means of a Green function: $\alpha = G\circ\beta$.
\end{proof}

This normalization is not mandatory, rather it is optional. The materials of the rest of the paper are presented without such normalization.

\section{Objects in the Riemannian Geometry}\label{sec:objects}

In this section we obtain the expressions of the main objects of the Riemannian geometry based on the set of 1-forms $\mathbf{A}$. These objects are the Levi-Civita connection and the curvature tensor. This last, is obtained from the first and second order ordinary derivative of the metric tensor. The first order derivative is:
\begin{equation}\label{eq:derivative_g}
\partial_{c} g_{ab} = \partial_{c}A_{a}\STAR A_{b} +
\partial_{c}A_{b}\STAR A_{a}
\end{equation}

\begin{definition}
Let $\mathbf{F}_{I}$ be a set of closed $2$-form defined as: $\mathbf{F}_{I} = d\mathbf{A}_{I}$. 
\end{definition}

The components of $\mathbf{F}_{I}$ are: $F_{Iab} = \partial_{a}A_{Ib}-\partial_{b}A_{Ia} = \nabla_{a}A_{Ib}-\nabla_{b}A_{Ia}$, expressed from the ordinary and the tensorial derivatives. The ordinary and tensorial derivative of $\mathbf{A}_{I}$ can be expressed based on $\mathbf{F}_{I}$ as:

\begin{equation}\label{eq:aux1}
\partial_{a}A_{Ib} = \partial_{(a} A_{Ib)}  + \partial_{[a} A_{Ib]} = \partial_{(a} A_{Ib)} + \frac{1}{2} F_{Iab}
\end{equation}

\begin{equation}\label{eq:aux2}
\nabla_{a}A_{Ib} = \nabla_{(a} A_{Ib)}  + \nabla_{[a} A_{Ib]} =   S_{Iab} + \frac{1}{2} F_{Iab}
\end{equation}

where $\partial_{(a} A_{Ib)}$, $\partial_{[a} A_{Ib]}$,  $S_{Iab}=\nabla_{(a} A_{Ib)}$ and $\nabla_{[a} A_{Ib]}$ are the symmetrical and skew-symmetrical parts of the ordinary and tensorial derivative respectively. Remark that the index $I$, which is concerning to set enumeration, is excluded of the operators for symmetry and skew-symmetry $()$ $[\,]$ respectively.

\begin{proposition}\label{prop:christoffel} 
The connections $\Gamma$ are expressed as:
\begin{eqnarray}
\Gamma_{cab} & = &  A_{c}\STAR \partial_{(a} A_{b)}  +
\frac{1}{2}\left [ A_{a}\STAR F_{bc} + A_{b}\STAR F_{ac} \right ] \\
\Gamma^{c}{}_{ab} & = & A^{c} \STAR \partial_{(a} A_{b)} +
\frac{1}{2}[A_{a} \STAR F_b{}^c + A_{b} \STAR F_a{}^c]
\end{eqnarray}
\end{proposition}
\begin{proof}
The result is obtained from on the Equation (\ref{eq:derivative_g}) and their definition\cite{Landau:1973, Misner:1973}:
\begin{equation}\label{eq:christoffeldef}
\Gamma_{cab}= \frac{1}{2}\left (\partial_{a} g_{cb} + \partial_{b} g_{ca} - \partial_{c} g_{ab} \right)
\end{equation}
\end{proof}

\begin{proposition}\label{prop:6}
The symmetric and skew-symmetric tensors $S_{Iab}$ and $F_{Iab}$ verify:
\begin{eqnarray}
A_{c} \STAR S_{ab} +
\frac{1}{2}[A_{a} \STAR F_{bc} + A_{b} \STAR F_{ac}] &=& 0
\end{eqnarray}
\end{proposition}
\begin{proof}
The symmetric part of the tensorial derivative can be expressed as:
\begin{equation}
 \nabla_{(a} A_{Ib)} =  \partial_{(a} A_{Ib)} - \Gamma^{c}_{ab} A_{Ic} = S_{Iab}
\end{equation}

that implies:
\begin{equation}
 \partial_{(a} A_{Ib)} =  S_{Iab} + \Gamma^{c}{}_{ab} A_{Ic}
\end{equation}

by multiplying it by $A_{Ic}$ and using the results of Proposition \ref{prop:christoffel} is obtained the proposed expression.
\end{proof}

\begin{proposition}\label{prop:7}
The symmetrical tensor $\mathbf{S}_{I}$ can be expressed from the skew-symmetric $\mathbf{F}_{I}$ as:
\begin{equation}
S_{Jab} = -\frac{1}{2}A_{J}^{c}[A_{a} \STAR F_{bc} + A_{b} \STAR F_{ac}]
\end{equation}
\end{proposition}
\begin{proof}
It is obtained by multiplying  the expression in the previous Proposition by $A_{J}^{c}$ and using the result of the Proposition \ref{prop:3}.
\end{proof}

It must be remarked that $\mathbf{S}_{I}$ is the symmetric derivative part of $\mathbf{A}_{I}$. However, it depends on the skew-symmetric derivative part $\mathbf{F}_{I}$, thus both derivative parts are dependents, being the skew-symmetric part that rules the symmetric one because if $\mathbf{F}_{I}$ is null also is null $\mathbf{S}_{I}$.

\begin{theorem}
If the set of $1$-forms $\mathbf{A}_{I}$ are closed, $\mathbf{F}_{}=d\mathbf{A}_{I}=0$, then the members this set are Killing vectors of the manifold: $S_{Iab}=(\nabla_{a} A_{Ib}+\nabla_{b} A_{Ia})/2=0$, and therefore the metric is invariant along its field lines, due to the Lie derivative: $\mathcal{L}_{\mathbf{A}_{I}}g_{ab}=0$
\end{theorem}

\begin{proposition}\label{prop:geodesic}
The geodesic line $\mathbf{u}$ of the manifold verifies:
\begin{equation}
\frac{du^{c}}{ds}+ A^{c} \STAR [\partial_{(a} A_{b)}u^{a}u^{b}] =
(A_{a}u^{a}) \STAR F^{c}{}_{b}u^{b}
\end{equation}
\end{proposition}
\begin{proof}
The geodesic line is defined as:
\begin{equation}\label{eq:geodesic}
\frac{Du^{c}}{ds}=\frac{du^{c}}{ds}+\Gamma^{c}_{ab}u^{a}u^{b}=0
\end{equation}

by substituting the expression of the connection and due to the symmetries:
\begin{equation}
\frac{du^{c}}{ds}+ A^{c} \STAR \partial_{(a} A_{b)}u^{a}u^{b} +
A_{a} \STAR F_b{}^c u^{a}u^{b}=0
\end{equation}
\end{proof}

According the result of this Proposition, even though the Lorentz-like right side, $F_{ab}u^{b}$,  be null the connection is non null and the geodesic are not straight lines.

\subsection{The Curvature Tensor}\label{subsec:curvature}

The curvature tensor can be obtained from two different procedures. The first is based on the successive ordinary derivatives of the metric tensor, while the second is based on the exterior derivative of the connection form $\boldsymbol\omega_{ab} = \Gamma_{abc} \mathbf{w}^{c}$. This second approach is more economic and elegant, but we will use the first approach based on the first and second derivative of the metric tensor. The Riemann curvature is defined from the connection as:\cite{Jost:2011,Misner:1973}:

\begin{equation}
R^{a}{}_{bcd} = \partial_{c}\Gamma^{a}{}_{bd} -
\partial_{d}\Gamma^{a}{}_{bc} +
\Gamma^{a}{}_{ec} \Gamma^{e}{}_{bd} - \Gamma^{a}{}_{ed} \Gamma^{e}{}_{bc}
\end{equation}

The expression of the curvature can be obtained in a local  reference system, that is by using the Equivalence Principle\cite{Misner:1973}, and next it can be generalized. In a local frame the connections become null and the curvature tensor is expressed as: $R_{abcd} =
\partial_{c}\Gamma_{abd} -
\partial_{d}\Gamma_{abc}$, or expressed from the second derivative of
the metric tensor as\cite{Landau:1973}:
\begin{equation}
R_{abcd} = \frac{1}{2}(\partial_{bc}g_{ad}+ \partial_{ad}g_{bc} -
\partial_{bd}g_{ac} - \partial_{ac}g_{bd})
\end{equation}

By grouping the terms involving the second and the first  derivatives as, $R_{abcd}= R^{(1)}_{abcd}+R^{(2)}_{abcd}$:

\begin{eqnarray}
2R^{(2)}_{abcd}  & = &
A_{a} \STAR (\partial_{bc}A_{d} - \partial_{bd}A_{c}) +
A_{b} \STAR (\partial_{ad}A_{c} - \partial_{ac}A_{d}) + \\ &&
A_{c} \STAR (\partial_{ad}A_{b} - \partial_{bd}A_{a}) +
A_{d} \STAR (\partial_{bc}A_{a} - \partial_{ac}A_{b}) \\
2R^{(1)}_{abcd} & = &
\partial_{b}A_{a} \STAR \partial_{c}A_{d} +
\partial_{b}A_{d} \STAR \partial_{c}A_{a} +
\partial_{a}A_{b} \STAR \partial_{d}A_{c} +
\partial_{a}A_{c} \STAR \partial_{d}A_{b}  \\ &&
-\partial_{b}A_{a} \STAR \partial_{d}A_{c}
-\partial_{b}A_{c} \STAR \partial_{d}A_{a}
-\partial_{a}A_{b} \STAR \partial_{c}A_{d}
-\partial_{a}A_{d} \STAR \partial_{c}A_{b}
\end{eqnarray}

\begin{definition}[Current and Pre-current]\label{def:current}
Let  $J_{Iabc}$  and $J_{Ia}$ be defined as:
\begin{equation}\label{eq:current}
J_{Iabc}= \nabla_{a}F_{Ibc} \qquad J_{Ib}= J^{a}{}_{Iab}=\nabla^{a}F_{Iab}
\end{equation}
named pre-current and current respectively. The pre-current is defined in a Riemannian frame as: $J_{Iabc}= \partial_{ab}A_{Ic}-\partial_{ac}A_{Ib}$.
\end{definition}

The current $\mathbf{J}_{I}$ is a 1-form that admit a more compact definition as: $\mathbf{J}_{I}=\delta \mathbf{F}_{I} =\Delta \mathbf{A}_{I} - d(\delta \mathbf{A}_{I})$. The normalization proposed in Proposition \ref{prop:gauge} is not mandatory, but if the 1-forms $\mathbf{A}_{I}$ are normalized, then:
\begin{equation}
\mathbf{J}_{I}=\Delta \mathbf{A}_{I}
\end{equation}

We have used the term pre-currents to name $\nabla_{a}F_{bc}$ because they seem be a primary magnitude from which can be obtained the currents by contraction. These tensors are very important in the curvature, Ricci and Einstein tensors as is shown afterward, therefore they must be relevantly considered in the study of the curved manifolds. They have two symmetries: $J_{Ia(bc)}=0$ and $J_{I[abc]}=0$ as illustrated in Figure \ref{fig:precurrents}.

\begin{figure}
  \centering
  \includegraphics[width=7cm]{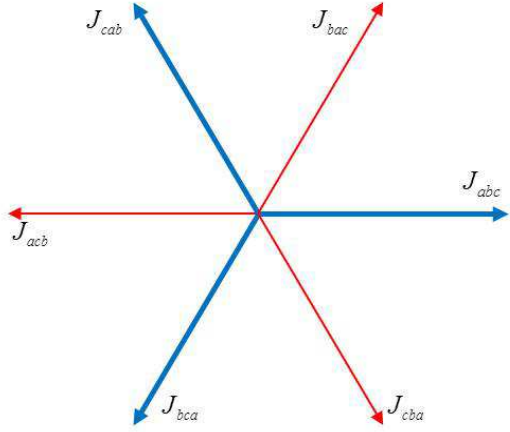}\\
  \caption{Map of the pre-current group with three different index into a 2-dimensional space. The two type of symmetries admit a graphic representation, the symmetries involving two pre-currents: $J_{abc}+J_{acb}=0$ and three pre-currents, the clockwise: $J_{abc}+J_{bca}+J_{cab}=0$ as well as the anti-clockwise: $J_{cba}+J_{bac}+J_{acb}=0$  }\label{fig:precurrents}
\end{figure}

The term in the curvature tensor involving the second derivatives can be  more compactly rewrite as:
\begin{equation}
2R^{(2)}_{abcd}  =
A_{a} \STAR  J_{bcd} - A_{b} \STAR  J_{acd} +
A_{c} \STAR  J_{dab} - A_{d} \STAR  J_{cab}
\end{equation}

The term involving the first derivative can be transformed by using the symmetric and skew-symmetric parts of the first derivative according the Equations (\ref{eq:aux1}) and (\ref{eq:aux2}) in the local Riemannian frame as:
\begin{eqnarray}
2R^{(1)}_{abcd} & = &
2S_{ac}\STAR S_{bd} - 2S_{ad}\STAR S_{bc}  \\ &&
-F_{ab}\STAR F_{cd}
- \frac{1}{2} F_{ac}\STAR F_{bd} + \frac{1}{2} F_{ad}\STAR F_{bc}
\end{eqnarray}

The expression in a general frame  can be carried out by the transformation: $\partial_{a} \rightarrow \nabla_{a}$. In this case, this transformation does not manifest explicitly.   The expression is:
\begin{eqnarray}\label{eq:curvature}
R_{abcd} &=& \frac{1}{2}(A_{a} \STAR  J_{bcd} - A_{b} \STAR  J_{acd}+
A_{c} \STAR  J_{dab} - A_{d} \STAR  J_{cab}) +\\
 && \frac{1}{4}(F_{ad} \STAR F_{bc} -
F_{ac} \STAR F_{bd} - 2 F_{ab}\STAR F_{cd}) \\
 && S_{ac}\STAR S_{bd} - S_{ad}\STAR S_{bc}
\end{eqnarray}

If the curvature tensor is expressed based on three sub-terms as: $R_{abcd}= R^{(c)}_{abcd} + R^{(f)}_{abcd} + R^{(s)}_{abcd}$ corresponding to every one of the previous lines:

\begin{eqnarray}\label{eq:curvature2}
R^{(c)}_{abcd} &=& \frac{1}{2}(A_{a} \STAR  J_{bcd} - A_{b} \STAR  J_{acd} +
A_{c} \STAR  J_{dab} - A_{d} \STAR  J_{cab}) \\
R^{(f)}_{abcd}  &=& \frac{1}{4}(F_{ad} \STAR F_{bc} -
F_{ac} \STAR F_{bd} - 2 F_{ab}\STAR F_{cd})\\
R^{(s)}_{abcd} &=& S_{ac}\STAR S_{bd} - S_{ad}\STAR S_{bc}
\end{eqnarray}

It is verified the first Bianchi identities for each one of the sub-terms; that is:

\begin{equation}
R_{a[bcd]}= R^{(c)}_{a[bcd]} = R^{(f)}_{a[bcd]} = R^{(s)}_{a[bcd]}=0
\end{equation}

\begin{theorem}
A sufficient condition for the manifold be non-curved, $R_{abcd}=0$, is that the set $\mathbf{A}_{I}$ be closed.
\end{theorem}
\begin{proof}
If $d\mathbf{A}_{I}=0$ implies that: $\mathbf{F}_{I}=d\mathbf{A}_{I}=0$, which implies: $\mathbf{S}_{I}=0$ and also $J_{Iabc}=0$, therefore $R_{abcd}=0$.
\end{proof}

The result of the previous Theorem generalizes the obtained in Theorem \ref{th:plainspace}, because be exact is a particular case of be closed. If $\mathbf{A}_{I}$ is closed it can be expressed according the Hodge decomposition as\cite{Morita:2001}:
\begin{equation}
\mathbf{A}_{I} = d\phi_{I} + \boldsymbol\gamma_{I}
\end{equation}

where $\boldsymbol\gamma_{I}$ is an harmonic $1$-form representative of the homology class to what  $\mathbf{A}_{I}$ belongs. However, it is an harmonic form, that is $\delta\boldsymbol\gamma_{I}=0$ and  $d\boldsymbol\gamma_{I}=0$ and it can not generate curvature. In simply connected manifold the condition of be closed is equivalent to be exact according the Poincar\'{e} Lemma, but in non-simply connected manifold, with homology classes, it is verified also: $\mathbf{F}_{I}=d\mathbf{A}_{I}=0$. The homology representative $1$-form $\boldsymbol\gamma_{I}$ is a closed but non-exact form because exists a cycle, a closed sub-manifold, $z$ such as\cite{Morita:2001}:
\begin{equation}
\int_{z}\boldsymbol\gamma_{I} \neq 0
\end{equation}

therefore there are flat manifolds that have not a globally Cartesian-like metric because the homology representative $1$-form can not be expressed as an exact form In this case, the line element is:
\begin{equation}
ds^{2} = \sum_{I=1}^{n} (d\phi_{I})^{2} + 2 d\phi\STAR \boldsymbol\gamma_{a} \mathbf{w}^{a} + \boldsymbol\gamma_{a}\STAR \boldsymbol\gamma_{b} \mathbf{w}^{a} \mathbf{w}^{b}
\end{equation}

the geodesic is according the Proposition \ref{prop:geodesic} and $\mathbf{F}_{I}=0$:
\begin{equation}
\frac{du^{c}}{ds}+ A^{c} \STAR \partial_{(a} A_{b)}u^{a}u^{b} =0
\end{equation}

that are not straight lines, except if $\partial_{(a} A_{b)}=0$ that implies that connections are null. 

\begin{theorem}
If the set $\mathbf{A}_{I}$ is closed but non-exact, that is, a non-simply connected and non-curved, $R_{abcd}=0$, manifold, then the metric is not reducible to a globally Cartesian-like, the connection is non null, and its geodesic are not straight lines.
\end{theorem}

The previous results suggest that is possible to define a subdivision of the flat manifold class in two subclasses: flat (or strong-flat) and semi-flat. The following classification summarizes the manifold types and subtypes, where the condition of non-closed 1-forms is associated to curved manifold and the condition of closed is associated to the two subtypes: strong-flat and semi-flat:
\begin{enumerate}
  \item Strong-flat manifold: the set of $1$-form are exact, $\mathbf{A}_{I}=d\phi_{I}$, its metric is globally reducible to  a Cartesian-like, both the connection and the curvature tensor are null, $\mathbf{A}_{I}$ are Killing vectors, and its geodesic are straight lines.
  \item Semi-flat manifold: the set of $1$-form are closed, $d\mathbf{A}_{I}=0$, but non exact, its metric is not globally reducible to  a Cartesian-like, the connection is non null, its geodesic are not straight lines, $\mathbf{A}_{I}$ are Killing vectors, and  the curvature tensor is null, $R_{abcd}=0$. 
  \item Curved manifold: the set of $1$-form are not closed, $d\mathbf{A}_{I}\neq 0$,  its metric is not globally reducible to  a Cartesian-like, neither the connection nor the curvature tensor are null, $R_{abcd}\neq 0$ and its geodesic are not straight lines.
\end{enumerate}

The criticism to this classification is that the strong-flat manifold is too much evidently flat. 
The introduction of refinements and nuances in the definition of flatness must reduce the number of manifold types included in this class\cite{Kobayashi:1996}[pp. 222], but always remains the extreme case, the Euclidean.

Figure \ref{fig:diagram} shows a diagram of the two analytical paths starting in the metric $g_{ab}$ and ending in the curvature $R_{abcd}$. The upper diagram is the usually used in the literature of Riemannian geometry\cite{doCarmo:1992}\cite{Jost:2011} where the metric, which is presented as the primary concept, the connection and curvature are the elements in a conceptual chain.
The lower diagram shows the approach proposed in this paper; the path from $g_{ab}$ to $R_{abcd}$ is achieved by using several intermediate geometric objects allowing an alternative viewpoint for the Riemannian geometry. The symmetrical tensor $\mathbf{S}_{I}$ is defined from $\mathbf{A}_{I}$, but in the practice it depends on the skew-symmetric $\mathbf{F}_{I}$, which is the cornerstone of this presentation of the Riemannian geometry obtained from the Takagi's factorization of the metric tensor. This approach is less economical because uses much more intermediate objects. However, allows a different viewpoint for the curvature as well as a new way for the decomposition of the curvature tensor.

\begin{figure}
  \centering
  \includegraphics[width=7cm]{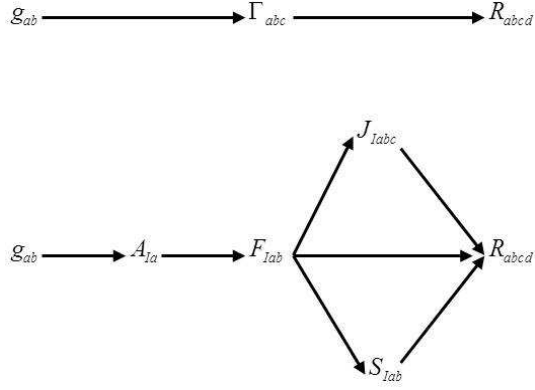}
  \caption{Diagrams of how the Riemannian curvature tensor can be obtained. The upper diagram is the the usual path contained in the Geometry literature. The lower diagram is the proposed in this paper based on a set of intermediate geometric objects.}\label{fig:diagram}
\end{figure}

\subsection{Ricci and Einstein Tensors}\label{subsec:Ricci}

The Ricci tensor, $R_{ab}=R^{c}{}_{acb}$, is expressed as:

\begin{eqnarray}\label{eq:ricci}
R_{ab} &=& \frac{1}{2}(-A_{a} \STAR  J_{b} + A^{c} \STAR  J_{acb} -
A_{b} \STAR  J_{a} + A^{c} \STAR  J_{bca}) +\\
 && -\frac{3}{4}F_{ac} \STAR F_{b}{}^{c} + S\STAR S_{ab}
  - S_{ac}\STAR S_{b}{}^{c}
\end{eqnarray}

where: $S_{I}= S^{a}_{Ia}=\nabla^{a} A_{Ia}$, is a scalar whose value is depending on the optional normalization of the $\mathbf{A}_{I}$ forms. The scalar curvature $R$ is:
\begin{equation}
R = -2 A_{a} \STAR J^{a} - \frac{3}{4}F_{ab} \STAR F^{ab} +S\STAR S - S_{ab}\STAR S^{ab}
\end{equation}

The Einstein Tensor $G_{ab}=R_{ab}-\frac{1}{2}g_{ab}R$ can be expressed as follows:
\begin{equation}\label{eq:newEinstein}
G_{ab} =  T^{(f)}_{ab} +  T^{(c)}_{ab} + T^{(s)}_{ab}
\end{equation}

where the three $\mathbf{T}$ symmetrical tensors are:
\begin{eqnarray}
T^{(f)}_{ab} &=& -\frac{3}{4} \left (F_{ac} \STAR F_{b}{}^{c} -\frac{1}{2} g_{ab} F_{cd} \STAR F^{cd}\right )\\
T^{(c)}_{ab} &=&  \frac{1}{2}\left(-A_{a} \STAR  J_{b} + A^{c} \STAR  J_{acb} -
A_{b} \STAR  J_{a} + A^{c} \STAR  J_{bca} + 2g_{ab}A_{c}\STAR J^{c}\right)\\
T^{(s)}_{ab} &=&   S\STAR S_{ab} -\frac{1}{2} g_{ab}S\STAR S - S_{ac}\STAR S_{b}{}^{c}  + \frac{1}{2}g_{ab} S_{cd}\STAR S^{cd}
\end{eqnarray}

The contracted second Bianchi identities implies that: $\nabla^{c}G_{ac}=0$, therefore it must be verified that:
\begin{equation}
\nabla^{b}\left(T^{(f)}_{ab} +  T^{(c)}_{ab} + T^{(s)}_{ab}\right) =0
\end{equation}

In a physic interpretation of the three $\mathbf{T}$  as the stress-energy tensor in the Einstein equation, the condition: $\nabla^{b} (\sum T_{ab})=0$ means that the $\sum T_{ab}$ is a compleat description of the matter fields. The Bianchi identity is verified due to the definition of the tensors, but its is interesting to illustrate how it is applied to some sub-terms. By resembling  a very similar term in Classical Field Theory\cite{Landau:1973,Misner:1973}, it is verified that:
\begin{equation}
\nabla^{b}\left ( F_{ac} \STAR F_{b}{}^{c} -\frac{1}{4} g_{ab} F_{cd} \STAR F^{cd}\right ) = J^{c^{}}\STAR F_{ac}
\end{equation}

but the most similar term in the Einstein tensor is the following, with some differences:
\begin{equation}
\nabla^{b} \left ( F_{ac} \STAR F_{b}{}^{c} -\frac{1}{2} g_{ab} F_{cd} \STAR F^{cd}\right )  =  J^{c}\STAR F_{ac}  -  J_{dca}\STAR F^{cd}
\end{equation}

\section{Conclusion}\label{sec:conclusion}

We have presented an analytic path to study the main objects of the Riemann geometry. As result of this study, these objects can be expressed based on a set $1$-forms as $\mathbf{A}_{I}$ and $\mathbf{J}_{I}$ ,$2$-form as $\mathbf{F}_{I}$, a symmetric rank 2 tensor $\mathbf{S}_{I}$ and a rank 3 tensor $J_{Iabc}$. This new analytical path is less economical, but provides some valuable results.

The curvedness or flatness property of a manifold depends on the curvature tensor, $R_{abcd}$, but this property can be alternatively  defined from the closed, or non-closed, property of the differential 1-form $\mathbf{A}_{I}$. This new condition may be more simple and also allows some refinements in the case of non-simply connected manifolds.

The aim of this paper is indirectly to provide solutions for the Einstein equation, but it involves two heterogeneous sides. The left side that concerns with the geometric $\mathbf{G}$, have a formal structure highly different to the physic right side, that concerns the physic $\mathbf{T}$, of mass and energy distributions. The problems in the solution for general cases may be highly dependent of the heterogeneous properties of both sides. Perhaps the lack of solutions that had involved many research from a century is due to this reason.

The proposed analysis provides a new factorization of the Einstein tensor in three sub-terms that resemble physic theories. The main conclusion of this paper is that if we can define a complete model of matter fields by means of stress-energy tensors fitting in these sub-terms, then the solution of the Einstein equation is immediate. Although the structures are similar to the Electromagnetic Field, they are not Maxwellian due to the importance of concepts as the pre-currents, the symmetric tensor and mainly the lack of phenomenological meaning because they are geometric objects. For physic applications, the four dimensional space-time that is modeled as a pseudo-Riemann manifold can be described by means of four vector fields that are very similar to the mathematic structure of Electromagnetic Fields. 


\bibliographystyle{plain}
\bibliography{bibliografiabasica}

\begin{thebibliography}{1}

\bibitem{doCarmo:1992}
M.P. do~Carmo.
\newblock {\em Riemannian Geometry}.
\newblock Birkh{\"a}user Boston, 1992.

\bibitem{Gockeler:1989}
M.~G\"{o}ckeler and T.~Sch{\"u}cker.
\newblock {\em Differential Geometry, Gauge Theories, and Gravity}.
\newblock Cambridge University Press, 1989.

\bibitem{Horn:1985}
R.A. Horn and C.R. Johnson.
\newblock {\em Matrix Analysis}.
\newblock Cambridge University Press, 1985.

\bibitem{Jost:2011}
J.~Jost.
\newblock {\em Riemannian Geometry and Geometric Analysis}.
\newblock Springer, 2011.

\bibitem{Kobayashi:1996}
S.~Kobayashi and K.~Nomizu.
\newblock {\em Foundations of Differential Geometry, Volume I}.
\newblock Wiley, 1996.

\bibitem{Landau:1973}
L.~D. Landau and E.M. Lifshitz.
\newblock {\em The Classical Theory of Fields}.
\newblock Butterworth-Heinemann, fourth ed. edition, 1973.

\bibitem{Misner:1973}
C.W. Misner, K.S. Thorne, and J.A. Wheeler.
\newblock {\em Gravitation}.
\newblock W. H. Freeman, 1973.

\bibitem{Morita:2001}
S.~Morita.
\newblock {\em Geometry of Differential Forms}.
\newblock American Mathematical Society, 2001.

\bibitem{Stephani:2009}
H.~Stephani, D.~Kramer, M.~MacCallum, C.~Hoenselaers, and E.~Herlt.
\newblock {\em Exact Solutions of Einstein's Field Equations}.
\newblock Cambdridge University Press, 2009.

\end{thebibliography}

\end{document}